\documentclass[10pt]{amsart}

\usepackage{amsmath,amssymb,amscd,amsfonts,amsthm,mathrsfs,eucal,dsfont,verbatim}

\usepackage{lscape,enumerate}
\usepackage{color}
\usepackage{setspace}

\usepackage{stackengine}

\usepackage{hyperref}

\usepackage{marginnote}
\newtheorem{theorem}{Theorem}
\newtheorem{assumption}{Assumption}

\newtheorem{condition}[theorem]{Condition}
\newtheorem{conjecture}[theorem]{Conjecture}
\newtheorem{corollary}[theorem]{Corollary}

\newtheorem{definition}[theorem]{Definition}
\newtheorem{example}[theorem]{Example}

\newtheorem{lemma}[theorem]{Lemma}

\newtheorem{proposition}[theorem]{Proposition}
\newtheorem{remark}[theorem]{Remark}

\newcommand{\real}{\mathds{R}}

\newcommand{\rd}{{\mathds{R}^d}}
\newcommand{\rdp}{{\mathds{R}_+^d}}

\newcommand{\supp}{\operatorname{supp}}
\newcommand{\Pois}{\operatorname{Pois}}
\newcommand{\Exp}{\operatorname{Exp}}

\newcommand{\Ee}{\mathds E}
\newcommand{\Pp}{\mathds P}
\newcommand{\I}{\mathds 1}
\newcommand{\Ff}{\mathcal{F}}
\newcommand{\Ll}{\mathcal{L}}
\newcommand{\Ss}{\mathcal{S}}
\newcommand{\WSs}{\mathcal{WS}}

\newcommand{\Bb}{\mathcal{B}}
\def\1{1\!\!\hbox{{\rm I}}}

\renewcommand{\Re}{\mathbb{R}}

\renewcommand{\Re}{\ensuremath{\operatorname{Re}}}
\renewcommand{\Im}{\ensuremath{\operatorname{Im}}}

\usepackage{color}

\usepackage{soul}

\usepackage[english]{babel}
\usepackage[left=3cm,right=3cm,top=2.5cm,bottom=3cm]{geometry}

\usepackage{amsmath,amssymb,amsthm,bbm,color,graphics,version}
\usepackage{mathrsfs}

\newcommand{\bearno}{\begin{eqnarray*}}
\newcommand{\enarno}{\end{eqnarray*}}

\newcommand{\ostar}{\mathbin{\mathpalette\make@circled\star}}

\title{Subexponential potential asymptotics with applications}

\author[V. Knopova]{Victoria Knopova  }
\address{TU Dresden, Zellescher Weg 12--14, Dresden, Germany}
\email{victoria.knopova@tu-dresden.de}
\author[Z. Palmowski]{Zbigniew Palmowski}
\address{ Faculty of Pure and Applied Mathematics,
Wroc\l aw University of Science and Technology,
Wyb. Wyspia\'nskiego 27, 50-370 Wroc\l aw, Poland}
\email{zbigniew.palmowski@pwr.edu.pl}

\thanks{This work is partially supported by Polish National Science Centre Grant No. 2018/29/B/ST1/00756, 2019-2022}

\date{\today}

\begin{document}
\numberwithin{equation}{section}
\begin{abstract}
Let  $X_t^\sharp$  be a multivariate process of the form $X_t =Y_t - Z_t$, $X_0=x$, killed at some  terminal time  $T$, where $Y_t$ is a Markov process having only jumps of the length  smaller than $\delta$, and $Z_t$ is a compound Poisson process with jumps of the length bigger than $\delta$ for some fixed $\delta>0$.
Under the assumptions that the summands in $Z_t$ are sub-exponential, we investigate the asymptotic behaviour of the potential function  $u(x)= \Ee^x \int_0^\infty  \ell(X_s^\sharp)ds$. The case of heavy-tailed entries in $Z_t$ corresponds  to the case of ``big claims'' in insurance models and  is of practical interest.
The main approach is based on fact that $u(x)$ satisfies a certain renewal equation.

\vspace{3mm}


\end{abstract}

\maketitle

\pagestyle{myheadings} \markboth{\sc V.\ Knopova --- Z.\ Palmowski
} {\sc Subexponential potential asymptotics with applications}




\section{Introduction}\label{sec:iar}

 Let $(X_t)_{t\geq 0}$ be a c\'adl\'ag   strong Markov process with values in $\rd$, defined  on the probability space $(\Omega, \mathcal{F}, (\mathcal{F}_t)_{t\geq 0},(\Pp^x)_{x\in \rd})$, where $\Pp^x (X_0=x)=1$, $(\mathcal{F}_t)_{t\geq 0}$ is a
right continuous natural filtration satisfying usual conditions and $\mathcal{F}:=\sigma(\bigcup_{t\geq 0} \mathcal{F}_t)$.

In this note we study the behaviour of the potential  $u(x)$  of the process $X$, killed at some terminal time, when the   starting point $x\in \rd$   tends infinity in the sense that $x^0\to \infty$, where $x^0:=\min_{1\leq i\leq d} x_i$.   A particular case of this model is the behaviour of the ruin probability if the initial capital $x$ is big. In the case when the claims  are heavy-tailed, this probability can still be quite large.  The other example where the function $u(x)$ appears comes from the mathematical finance, where $u(x)$ describes the discounted utility of consumption; see
see \cite{AsmAlbr, Mi97, tomek}  and references therein.
We show that in some cases one can still calculate the asymptotic behaviour of $u(x)$ for large $x$, and discuss some practical examples.

Let us introduce some necessary notions and notation.
Assume
that $X$ is of the form
\begin{equation}\label{X1}
X_t := Y_t -Z_t,
\end{equation}
where $Y_t$ is a c\'adl\'ag, $\rd$-valued strong  Markov   process
with jumps of size strictly smaller than some  $\delta>0$, and $Z_t$ is an independent of $Y_t$ compound Poisson process with jumps of the size  bigger  than $\delta$.
That is,
\begin{equation}\label{CP}
Z_t:= \sum_{k=1}^{N_t}U_k,
\end{equation}
where $\{U_k\}$ is a sequence of i.i.d. random variables with a distribution function $F$,
\begin{equation}\label{bounddelta}
|U_k|\geq \delta,\qquad k\geq 1
\end{equation}
and $N_t$ is an independent Poisson process with intensity
$\lambda$.  In this set-up we have $\Pp^x (Y_0=x)=1$.

 Let $T$ be an $\Ff_t$-terminal time,  i.e.   for any $\Ff_t$-stopping  time $S$  it satisfies the relation
\begin{equation}\label{ter}
S+ T\circ \theta_S= T \quad \text{on $\{S<T\}$}.
\end{equation}
see \cite[$\S12$]{Sh88} or \cite[$\S22.1$]{Ba06}. 
 Among the examples of terminal times are

\begin{itemize}
\item The first exit time $\tau_D$ from  a Borel set  $D$: \,  $\tau_D:= \inf\{t>0: \, X_t \notin D\}$;

\item The exponential  (with  some parameter $\mu$)   random variable independent of $X$;

\item $T:= \inf\{t>0:\, \int_0^t f(X_s)\, ds\geq 1\}$, where $f$ is  a non-negative function;

\end{itemize}
see \cite{Sh88} for more examples.

  For $t\geq 0$ we define the killed process
\begin{equation}\label{kill}
X^\sharp_t:= \left\{
\begin{array}{lr}
X_{t}& t<T\\
\partial&t\geq T,
\end{array}
\right.
\end{equation}
where $\partial$
is a fixed cemetery  state. Note that the killed process $(X^\sharp_t, \Ff_t)$ is still strong Markov (cf. \cite[Prop. 22.1]{Ba06}).

\medskip

 Denote by  $\Bb_b(\rd)$ (resp.,  $\Bb_b^+(\rd)$)   the class of bounded (resp.,
bounded such that the infimum is nonegative on $\rd$ and it is strictly positive on $\rdp$)
Borel  functions on $\rd$.

We investigate the asymptotic properties of  the potential of $X^\sharp$:
\begin{equation}\label{u10}
u(x):= \Ee^x \int_0^\infty  \ell(X_s^\sharp)ds= \int_0^\infty  \Ee^x[\ell(X_s)\I_{T>s} ] ds, \quad x\in \rd,
\end{equation}
where    $\ell \in \Bb_b^+(\rd)$ and
through the paper we assume that
$\ell(\partial)=0$.
From this assumption $\ell(\partial)=0$ we have $u(\partial)=0$.
This function $u(x)$ is a particular example of a Gerber-Shiu function (see \cite{AsmAlbr}), which relates the ruin time and the penalty function, and appears often in the insurance mathematics when one needs to calculate the risk of a ruin.
We assume that the function $u(x)$ is well-defined and  bounded.
For example this is true if
$\Ee^x T=\int_0^\infty \Pp^x (T>s)ds<\infty$
because $\ell \in \Bb_b^+(\rd)$.

Having  appropriate   upper  and lower bounds\,    on the  transition probability density of $X_t$ it is possible to estimate $u(x)$. However, in some cases  one can get the  asymptotic behaviour of $u(x)$.
In fact, using the
strong Markov property,  one can show that  $u(x)$ satisfies the following  renewal type equation
\begin{equation}\label{ren1}
 u(x)= h (x)+   \int_\rd  u(x-z) \mathfrak{G}(x,dz),
 \end{equation}
 with some $h\in\mathcal{B}_b^+(\rd)$ and a (sub-)probability measure $\mathfrak{G}(x,dz)$ on $\rd$ identified explicitly. Note that under the assumptions made above this equation has a unique bounded solution (cf. Remark~\ref{rem2}).
 In the case when $Y_t$ has independent increments, this is a typical renewal equation, i.e. \eqref{ren1} becomes
 \begin{equation}\label{ren11}
 u(x)= h (x)+ \int_\rd u(x-z) G(dz),
 \end{equation}
for some (sub-)probability measure $G(dz)$.

In the case when $T$ is an independent killing, the measure  $\mathfrak{G}(x,dz)$ is a sub-probability measure with $\rho:=\mathfrak{G}(x,\rd)<1$ (note that $\rho$ does not depend on $x$, see \eqref{q01} below).  This makes it possible to give precisely the asymptotic behaviour of $u$   if $F$ is ($\rd$)-sub-exponential. The case when $F$ is sub-exponential corresponds to the situation when the impact of the claim is rather big, e.g., $U_i$ does not have finite variance. Such a situation appears in may insurance models, see, for example, Mikosch \cite{Mi97} as well as the monographs Asmussen \cite{As03}, Asmussen, Albrecher \cite{AsmAlbr}. We discuss several practical examples in Section~\ref{sec:app}.

The case when the time $T$ depends on the process might be different though.
We discuss this problem in  Example~\ref{risk}, where $X$ is a
one-dimensional risk process with $Y_t= at$, $a>0$, and $T$ is a ruin time, that is, the first time when the process gets below zero.
In this case we suggest to rewrite equation \eqref{ren11} in a different way in order to deduce the asymptotic of $u(x)$.

  Asymptotic behavior of the solution to the renewal equation of type \eqref{ren1}  has been studied quite a lot, see the monograph of  Feller \cite{Fe71} and  also \c{C}inlar \cite{Ci69}, Asmussen \cite{As03}. The behaviour of the solution heavily depends on the integrability of $h$ and the behaviour of  the tails of $G$. We refer to \cite{Fe71}  for the classical situation when the Cram\'er-Lundberg condition holds true, i.e. when there exists a solution $\alpha=\alpha(\rho, G)$ to the equation $\rho\int e^{\alpha x} G(dx)=1$, see also Stone \cite{St65}  for a moment condition. In the multi-dimensional case under the generalization of the Cram\'er-Lundberg  or   moment assumptions,  the asymptotic behaviour of the solution was studied in  Chung \cite{Ch52}, Doney \cite{Do66}, Nagaev \cite{Na79},    Carlsson, Wainger \cite{CW82,CW84},  H\"oglund \cite{H88}   (see also the reference therein for multi-dimensional renewal theorem).  In Chover, Nei, Wainger \cite{CNW73a,CNW73b} and Embrecht, Goldie \cite{EG80,EG82} the asymptotic behaviour of the tails of the measure $\sum_{j=1}^\infty c_j G^{*j}$ on $\real$  was investigated under the subexponentiality condition on the  tails of $G$, e.g.  when the moment condition is not necessarily satisfied. These results were further extended in the works of Cline \cite{Cl86,Cl87}, Cline and Resnik \cite{CR92},  Omey \cite{Om06},  Omey, Mallor, Santos \cite{OMS06},  Yin, Zhao \cite{YZ06},  see also the monographs of Embrechts, Kl\"uppelberg, Mikosh \cite{EKM97}, Foss,  Korshunov, Zahary \cite{FKZ13}.

The main tools used in this paper to derive the  above mentioned asymptotics of the potential $u(x)$ given in \eqref{u10} are based
on the properties of subexponential distributions in $\rd$ introduced and discussed in \cite{Om06, OMS06}.

The paper is organized as follows. In Section \ref{sec:renewal} we construct the renewal equation for the potential function $u$.
In Section \ref{sec:as} we give main results.
Some particular examples and extensions are described in Section \ref{sec:examples}. Finally, in Section \ref{sec:app} we
give some possible applications of proved results.

\medskip

 We use  the following notation. We write $f (x)\asymp g(x)$ when $C_1g(x)\leq f(x)\leq C_2g(x)$ for some constants $C_1, C_2\leq 0$. We write $y<x$ for $x,y\in \rd$,  if all components of $y$ are less than respective components of $x$.

\section{Renewal type  equation: general case}\label{sec:renewal}

Let $\zeta\sim \Exp(\lambda)$  be the moment of the first big jump of size $\geq \delta$  of the process $Z_t$.
Denote
\begin{equation}\label{h12}
h(x): = \Ee^x \int_0^\zeta \ell(X_s^\sharp)ds  =\int_0^\infty  e^{-\lambda r}   \Ee^x  [ \ell(Y_r)\I_{T>r} ]    dr.
\end{equation}
 For a Borel measurable set $A\subset \rd$
\begin{equation}\label{G110}
 \mathfrak{G}(x,A) : =\Ee^x [F(A+ Y_\zeta-x)\I_{T>\zeta}].
\end{equation}
 In the case when $Y_s$ is not a deterministic function of $s$, the kernel  $\mathfrak{G}(x,dz)$ can be rewritten in the following way:
\begin{equation}\label{G11}
\mathfrak{G}(x,dz) : = \int_0^\infty  \int_\rd  \lambda e^{-\lambda s}   F(dz+ w)\Pp^x(Y_s\in  dw+x,T>s)ds.
\end{equation}

In the theorem below  we derive the renewal (-type) equation for $u$.

For the kernels $H_i(x,dy)$, $i=1,2$ define the convolution
\begin{equation}\label{H12}
(H_1* H_2)(x,dz):= \int_\rd H_1(x-y, dz-y)H_2(x,dy).
\end{equation}
Note that if $H_i$ are of the type $H_i(x,dy)= h_i(y)dy$, $i=1,2$, then this convolution reduces to the ordinary convolution of the functions $h_1$ and $h_2$:
$$
(H_1* H_2)(x,dz):= \left(\int_\rd   h_1(z-y) h_2(y)dy \right)dz.
$$
Similarly, if only $H_1(x,dy)$ is  of the form $H_1(x,dy)= h_1(y)dy$, then by $(h_1* H_2)(z,x)$ we understand
$$
(h_1* H_2)(z,x)= \int_\rd   h_1(z-y) H_2(x,dy).
$$
\begin{theorem}\label{t-ren}
Assume that  the terminal time satisfies $\Ee T<\infty$.
Then the function $u(x)$  given by  \eqref{u10} is a solution to the equation \eqref{ren1}, and admits the representation
\begin{equation}\label{u1}
u(x)=  \Big(h* \sum_{n=0}^\infty \mathfrak{G}^{ *n }(x,\cdot) \Big)(x,x),
\end{equation}
where   $\mathfrak{G}^{* 0 }(x,dz)=\delta_0(dz)$ and  $\mathfrak{G}^{* n }(x,dz) := \int_\rd \mathfrak{G}^{* (n-1) }(x,dy)\mathfrak{G}(x-y,dz-y)$ for $n\geq 1$.
\end{theorem}
 If $Y_t$ has independent increments, then
\begin{equation}\label{G1}
\mathfrak{G}(x,dz) \equiv  G(dz)  =\int_0^\infty \lambda e^{-\lambda s}  \int_\rd F(dz+w) \Pp^0 (Y_s\in dw, T>s)ds
\end{equation}
and
\begin{equation}\label{u2}
u(x)=  \Big(h* \sum_{n=0}^\infty   G^{* n }\Big)(x,x).
\end{equation}

\begin{remark}\label{rem2}\rm
Recall that $u(x)$ is assumes to be bounded.
Then,  since  $\ell\in \mathcal{B}_b^+(\real)$, $u(x)$ is the unique bounded solution to \eqref{ren1}.
The proof of this fact is similar to that in Feller \cite[XI.1, Lem. 1]{Fe71}. Indeed, suppose that $v(x)$ is another bounded solution to \eqref{ren1}. Take $x\in \rd\backslash \partial$. Then $w(x):= u(x)-v(x)$ satisfies the equation
$$
w(x)= \big(w* \mathfrak{G}(x,\cdot )\big)(x,x) =\big(w* \mathfrak{G}^{* 2} (x,\cdot )\big)(x,x) = \dots= \big(w* \mathfrak{G}^{ * n} (x,\cdot )\big)(x,x) , \quad n\geq 1.
$$
 Note that for any   Borel measurable $A\subset \rd $  we have by \eqref{G110}  $ \mathfrak{G} (x,A)  < 1$. Then
$$
\max_{y\in A} |w(y)|\leq \max_{y\in A} |w(y)| \, \mathfrak{G}^{* n} (x,A) \to 0, \quad  \text{as $n\to \infty$}.
$$
Hence, $w(x)\equiv 0$ for $x\in A$ for any $A$ as above.
\end{remark}
   Before we proceed to the proof of Theorem~\ref{p2}, recall the definition of the strong Markov property, which is crucial for the proof. Recall (cf. \cite[\S2.3]{Ch82}) that the process $(X_t,\mathcal{F}_t)$ is called strong Markov, if for any optional time $S$
 and  any  continuous on $\overline{\rd}:=\rd\cup \{\infty\}$ real-valued function $f$ such that $\underset{x\in\overline{\rd}}{\sup}|f(x)|<\infty$,
 \begin{equation}\label{sm}
\Ee^x f(X_{S+r} |\Ff_{S}) = \Ee^{X_S} f(X_r),\qquad r\geq 0.
\end{equation}
Here $\Ff_{S}:= \{ A\in \Ff| \, A \cap \{S\leq t\}\in \Ff_{t+}\equiv\Ff_{t}\,\quad \forall t\geq 0\}$, and since $\Ff_t$ is assumed to be right-continuous, the notions of the stopping and optional times coincide.
Sometimes it is convenient to reformulate the strong Markov property in terms of the shift operator: let  $\theta_t: \Omega\to \Omega$ be such that
 for all $r>0$ $(X_r \circ \theta_t)(\omega) = X_{r+s}(\omega)$. This operator naturally extends  to $\theta_S$ for an optional time $S$ as  follows: $(X_r \circ \theta_S)(\omega) = X_{r+S}(\omega)$. Then one can rewrite \eqref{sm} as
  \begin{equation}\label{sm2}
\Ee^x [f(X_r\circ \theta_S) |\Ff_{S}] = \Pp^{X_S} f(X_r).
\end{equation}
and for any $Z\in \Ff$
  \begin{equation}\label{sm22}
\Ee^x [Z \circ \theta_S  |\Ff_{S}] = \Ee^{X_S}  Z \quad \Pp^x - \text{a.s. on $\{S<\infty\}$}.
\end{equation}
Definition \eqref{ter} of the terminal time  $T$ allows to use the  strong Markov   property  \eqref{sm22} in order to  ``separate'' the future of the process from its past.

\begin{proof}[Proof of Theorem~\ref{t-ren}]
Using the strong  Markov property  we get
$$
u(x)= \Ee^x \left[ \int_0^\zeta + \int_\zeta^\infty\right]  \ell(X_s^\sharp)ds := I_1+I_2.
$$
We estimate both terms  $I_1$ and $I_2$ separately.
Note that  $X_s^\sharp= Y_s^\sharp$ for $s\leq \zeta$. Therefore by the Fubini theorem we have
\begin{align*}
I_1 &= \Ee^x \int_0^\infty \lambda e^{-\lambda s} \int_0^s \ell(Y_r^\sharp)\, dr ds= \Ee^x \int_0^\infty \left( \int_r^\infty \lambda e^{-\lambda s} \, ds \right) \ell (Y_s^\sharp) \, dr\\
&= \Ee^x \int_0^\infty e^{-\lambda r} \ell(Y_r^\sharp)\, dr= \int_\rd \ell(w) \int_0^\infty e^{-\lambda r} \Pp^x (Y_r \in dw, T>r)dr\\
&=h(x).
\end{align*}
To transform $I_2$ we use that $T$ is the terminal time, the strong Markov  property \eqref{sm22} of $X$, and that $X_\zeta^\sharp= Y_\zeta^\sharp$. Let $Z= \int_0^\infty  \ell(X_r^\sharp) \,dr$. Then by the definition \eqref{ter} of  the terminal time we get
\begin{align*}
I_2&= \Ee^x \int_0^\infty \ell(X_r^\sharp\circ \theta_\zeta)dr=\Ee^x\left[ \Ee^x \left[\int_0^\infty  \ell(X_r^\sharp\circ \theta_\zeta) \,dr \Big| \Ff_\zeta \right]\right] \\
&= \Ee^x \left[ \Ee^x \left[Z \circ \theta_\zeta \Big| \Ff_\zeta \right] \right]= \Ee^x [\Ee^{X_\zeta^\sharp}  Z]= \Ee^x u(X_\zeta^\sharp)=\Ee^x u(Y_\zeta^\sharp) \\
 &= \int_\rd   \int_\rd   u(w-y)  \left[\int_0^\infty   \lambda e^{-\lambda s } F(dy) \Pp^x( Y_s \in dw, T>s) ds\right] \\
&=  \int_\rd \int_\rd    u(v-(y-x))\left[ \int_0^\infty  \lambda e^{-\lambda s }  F(dy) \Pp^x( Y_s \in dv+x, T>s) ds\right]\\
&= \int_\rd   u(x-z)  \left[\int_\rd   \int_0^\infty \lambda e^{-\lambda s }   F(dz+v) \Pp^x( Y_s \in dv+x, T>s) ds\right].
\end{align*}
where  in the third and the last lines from below we used  the Fubini theorem, and in the last two lines we made the change of variables $w\rightsquigarrow v+x$ and $y\rightsquigarrow v+z$, respectively. The integral in the square brackets in the last line is equal to $\mathfrak{G}(x,dz)$.
Thus  $u$ satisfies the renewal equation \eqref{ren1}. Iterating this equation we get \eqref{u1}.
\end{proof}

\section{Asymptotic behavior in case of independent killing}\label{sec:as}

In this section we show that under certain conditions one can get the asymptotic behaviour of $u(x)$ for large $x$.
 We begin with a small sub-section where we collect the  necessary auxiliary  notions.

\subsection{Sub-exponential  distributions on $\rdp$ and $\rd$}
 Recall the notation $\rdp= (0,\infty)^d$    and
$x^0=\min_{1\leq i\leq d} x_i<\infty$ for $x\in \rd$.

 \begin{definition}\label{defn}
\begin{enumerate}
\item  A function $f: \rdp\to [0,\infty)$ is called weakly long-tailed (notation: $f\in WL(\rdp)$) if
\begin{equation}\label{w1}
\lim_{x^0\to \infty} \frac{f(x-a)}{f(x)} =1 \quad \forall a>0.
\end{equation}

\item We say that a distribution  function $F$ on $\rdp$ is  \emph{weakly subexponential} (notation: $F\in WS(\rdp)$) if
\begin{equation}\label{subexp}
\lim_{x^0\to \infty}\frac{\overline{F^{*2}}\big(x\big)}{\overline{F}\big(x\big)} =2.
\end{equation}
\item We say that a distribution  function $F$ on $\rd$ is weakly subexponential (notation: $F\in WS(\rd)$) if it it long-tailed and \eqref{subexp} holds true.
\end{enumerate}
\end{definition}

\begin{remark}\rm
\begin{enumerate}
\item If $F\in WS(\rdp)$  then $\overline{F}$ is long-tailed.

\item  For $F\in WS(\rdp)$ we have (cf. \cite[Cor. 11]{Om06})
 \begin{equation*}
\lim_{x^0\to \infty}\frac{\overline{F^{*n}}\big(x\big)}{\overline{F}\big(x\big)} =n.
\end{equation*}

\item  Rewriting \cite[Lem. 2.17, p. 19]{FKZ13} in a multivariate set-up we  conclude
that any weakly subexponential distribution function is heavy-tailed, that is, for any $\varsigma >0$,
\begin{equation}\label{heavy}
\lim_{x^0\to \infty}\overline{F}(x)e^{\varsigma x}=+\infty.
\end{equation}
\item  We extended the definition of the whole-line subexponentialy from \cite[Def. 3.5]{FKZ13} to the multi-dimensional case. Note that even on the real line the assumption \eqref{subexp} along does not imply that the distribution is long-tailed, see \cite[$\S$ 3.2]{FKZ13}.
\end{enumerate}
\end{remark}

An important property  of a long-tailed function $f$  is the existence  of an insensitive function.
\begin{definition}\label{insen}
We say that a function $f$ is $\phi$-insensitive as $x^0\to\infty$, where  $\phi:\rdp\to \rdp$  is a non-negative increasing in each coordinate function,  if $\lim_{x^0\to \infty} \frac{f(x+\phi(x))}{f(x)}=1$.
\end{definition}
\begin{remark}\rm
 Observe that taking in Definition \ref{insen} $x-\phi(x)$ instead of $x$ and assuming  that $(x-\phi(x))^0 \rightarrow +\infty$
provides
that for $\phi$-insensitive function $f$ we have
$\lim_{x^0\to \infty} \frac{f(x-\phi(x))}{f(x)}=1$ as well.
\end{remark}
\begin{remark}\label{phi10}\rm
In the one-dimensional case if $f$ is long-tailed then such a function $\phi$ exists, and if $f$ is regularly varying,  then it is $\phi$-insensitive with respect to  any function $\phi(t)= o(t)$ as $t\to \infty$. The observation below shows that this property can be extended to the multi-dimensional case.

 Let $\phi(x)= (\phi_1(x_1), \dots, \phi_d(x_d))$, where $\phi_i : [0,\infty)\to [0,\infty)$, $1\leq i\leq d$,  are increasing functions, $\phi_i(t)= o(t)$ as $t\to \infty$. If $f$ is regularly varying in each component (and, hence,  long-tailed in each component),  then it is  $\phi(x)$-insensitive. Indeed,
\begin{equation}\label{w4}
\begin{split}
\lim_{x^0 \to \infty} \frac{f(x+\phi(x))}{f(x)} &= \lim_{x^0 \to \infty}\left\{ \frac{f(x_1+ \phi_1(x_1), \dots, x_d+ \phi_d(x_d))}{f(x_1+\phi_1(x_1), \dots, x_{d-1}+ \phi_{d-1}(x_{d-1}), x_d)} \cdot   \right.\\
&\quad \cdot \left.\frac{f(x_1+ \phi_1(x_1), \dots,x_{d-1}+ \phi_{d-1}(x_{d-1}),  x_d)}{f(x_1+\phi_1(x_1), \dots, x_{d-1}, x_d)}
  \dots  \frac{f(x_1+ \phi_1(x_1),  x_2,\dots, x_d)}{f(x_1, x_2, \dots,  x_d)}\right\}
 \\
&=1.
\end{split}
\end{equation}
\end{remark}
\begin{remark}\rm
Note that if a function is regularly varying in each component, it is long-tailed in the sense of definition \eqref{w1}, which follows from  \eqref{w4}. However, the  class of long-tailed functions is larger than that of multivariate regularly varying functions. There are several definitions of multivariate regular variation, see e.g. \cite{BDM02,  Om06}.
According to  \cite{Om06},
a function $f: \rdp\to [0,\infty)$ is called regularly varying, if for any $x\in \rdp$
\begin{equation}\label{w2}
\lim_{t\to \infty} \frac{f(tx-a)}{t^{-\kappa} r(t)} = \psi(x),
\end{equation}
where $\kappa \in \real$,  $r(\cdot)$ is slowly varying at infinity  and $\psi(\cdot)\geq 0$   (see \cite{BDM02} for  the definition of a multivariate regular variation of a distribution tail),  or   weakly regularly varying w.r.t. $h$, if for any $x,b\in \rdp$,
\begin{equation}\label{w22}
\lim_{b^0\to \infty} \frac{f(bx-a)}{h(b)} =\psi (x),
\end{equation}
where $b x:=(b_1x_1,\dots, b_dx_d)$.
Note that the function of the form $f(x_1,x_2)= c_1(1+x_1^{\alpha_1})^{-1} + c_2 (1+x_1^{\alpha_2})^{-1}$
($c_i, \alpha_i>0$, $i=1,2$) is regularly varying in each variable, but is not regularly varying in the sense of \eqref{w2} or \eqref{w22} unless $\alpha_1=\alpha_2$.
\end{remark}

\subsection{Asymptotic behaviour of $u(x)$}

 Let $T$ be an independent exponential killing  with parameter $\mu$.    We assume that the  law  $P_s(x,dw)$ of  $Y_s$  is absolutely  continuous with respect to the Lebesgue measure,  and denote the respective transition probability density function by  $\mathfrak{p}_s(x,w)$.

Rewrite $\mathfrak{G}(x,dz)$ as
\begin{equation}\label{q1}
\mathfrak{G}(x,dz)  = \int_\rd  F(dz+w)  q(x,w+x) dw,
\end{equation}
were
 \begin{equation}\label{additional}
q(x,w):= \int_0^\infty  \lambda e^{-\lambda s} \Pp(T>s)  \mathfrak{p}_s (x,w) ds.
\end{equation}
Observe that in the case of independent killing we  have (cf. \eqref{q1})
\begin{equation}\label{q01}
\sup_x \mathfrak{G}(x,\rd)= \int_0^\infty \lambda e^{-\lambda s} \Pp(T>s) \, ds = \rho:=\frac{\lambda}{\lambda+\mu}<1.
\end{equation}
 For $z\in \rd$, denote
$$
G_\rho(x,z):=\rho^{-1}  \mathfrak{G}(x,(-\infty, z]).  
$$

%
%
%
\begin{theorem}\label{p2}
 Assume that $T$ is an independent exponential killing  with parameter $\mu$
and $\ell(x)\to 0$ as $x^0
\to -\infty$. 
Let $F\in WS(\rdp)$  and suppose that  the function $q(x,w)$ defined in \eqref{q1} satisfies the estimate
\begin{equation}\label{q22}
q(x,w) \leq C e^{-\theta |w-x|}
\end{equation}
for some $\theta, C>0$.  Suppose that $\ell$  is long tailed and $\phi$-insensitive for some  $\phi$ such that  $\phi(x)^0\rightarrow +\infty$ and   $(x-\phi(x))^0\rightarrow +\infty$ as $x^0\to \infty$,  and   for any $c>0$
\begin{equation}\label{lp}
\lim_{x^0\to \infty}\min(\overline{F}(x),  \ell(x)) e^{c|\phi(x)|}= \infty
\end{equation}
 Suppose also that there exists $B\in [0,\infty]$ such that
 \begin{equation}\label{limit}
 \lim_{x^0\to\infty} \frac{\ell\big(x\big)}{\overline{F}(x) }=B.
 \end{equation}
 If $B=\infty$ we assume in addition that    $\ell(x)$ is regularly varying in each component.
Then
\begin{equation}\label{ub}
u(x) =
\begin{cases}
\frac{B \rho }{1-\rho}\overline{F}(x) (1+ o(1)), &  B \in (0,\infty),\\
o(1) \overline{F}(x), & B =0,\\
\frac{ \rho \ell(x)}{1-\rho}  (1+ o(1)), & B=\infty,
\end{cases}
\quad \text{as $x^0\to \infty.$}
\end{equation}
%
%
%
%
\end{theorem}

\begin{remark}\rm
In the one-dimensional case   and $Y_t$ being a L\'evy process the proof follows from \cite[Cor. 3]{EGV79},  \cite[Thm. A.3.20]{EKM97}, or \cite[Cor. 3.16-3.19]{FKZ13}.
\end{remark}

\begin{remark}\rm
One can relax the condition of existence of  the limit \eqref{limit}  replacing it by the existence of $\underset{x^0\to\infty}{\limsup}$ and $\underset{x^0\to\infty}{\liminf}$,
and  the  assumption that $\ell$ is regularly varying  in each  component by
$$
0<c<\liminf_{x^0 \to \infty}\frac{\ell (x+w)}{\ell (x)}\leq \limsup_{x^0 \to \infty}\frac{\ell (x+w)}{\ell (x)}\leq C.
$$
Since this extension is straightforward, we  do not go into details.
\end{remark}
\begin{remark}\label{Remarksuffcond}\rm
By \eqref{q22} and the dominated convergence theorem, the assumption $\ell(x)\to 0$ as $x^0\to -\infty$ implies that $h(x)\to 0$ as $x^0\to -\infty$.
\end{remark}

 For the proof of Theorem~\ref{p2}   we need the following auxiliary lemmas.
\begin{lemma}\label{GF2}  Under the assumptions of Theorem~\ref{p2} we have
\begin{equation}\label{Gnew1}
\lim_{x^0\to \infty}\sup_z \frac{\overline{G_\rho^{* n}}(z,x) }{\overline{F}(x)} =\lim_{x^0\to \infty}\inf_z \frac{\overline{G_\rho^{* n}}(z,x) }{\overline{F}(x)} = \lim_{x^0\to \infty}\frac{\overline{G_\rho^{* n}}(z,x) }{\overline{F}(x)}= n, \quad n\geq 1,
\end{equation}
and there exists $C>0$ such that
\begin{equation}\label{Gnew2}
\lim_{x^0\to \infty} \sup_z  \frac{\overline{G^{* n}_\rho}(   z,x)}{\overline{F}(x)} \leq C n (1+\epsilon)^n.
\end{equation}
\end{lemma}
\begin{proof}
The proof is similar to that of \cite[Thm. 3.34]{FKZ13}.
The idea is that the parametric dependence on $x$ is hidden in the function $q(x,x+w)$,  which decays much faster than $\overline{F}$ because of \eqref{heavy}.

Take $\phi$ such that $\overline{F}$ is $\phi$-insensitive
and $(x-\phi(x))^0\rightarrow +\infty$.

We split:
\begin{align*}
\overline{G}_\rho(z,x)&= \rho^{-1}\int_\rd \overline{F}(x+w) q(z,w+z)dw\\
&= \rho^{-1}\left( \int_{  w \leq -\phi(x)  }  +\int_{ |w| \leq |\phi(x)|  }  + \int_{w > \phi(x)} \right) \overline{F}(x+ w) q(z,w+z)dw \\
&:= K_1(z,x)+ K_2(z,x)+K_3(z,x).
\end{align*}
 We have by \eqref{q22}
\begin{equation}\label{K1}
\sup_z K_1(z,x) \leq \rho^{-1} \int_{w< - \phi(x)}  q(z,w+z)dw \leq  C_1 \int_{w< - \phi(x)}  e^{-\theta|w|} dw\leq C_2 e^{-\theta |\phi(x)|}
\end{equation}
and
\begin{equation}\label{K3}
\sup_z K_3(z,x)\leq  C_3 \int_{v \geq \phi(x)} e^{-\theta|v|}dv\leq C_4 e^{-
\theta |\phi(x)|}.
\end{equation}
From \eqref{lp} it follows that the left-hand sides of the above inequalities are  $o(\overline{F}(x))$ as $x^0 \to \infty$.

Note that  $K_2(z,x)\leq \sup_{|w|\leq \phi(x)} \overline{F}(x -w)$. Hence by Definition \ref{insen}, Remark
\ref{phi10}
and $\phi$-insensitivity of $\overline{F}$
we can conclude that

$$
\lim_{x^0\to \infty} \sup_z\frac{K_2(z,x) }{\overline{F}(x)}=\lim_{x^0\to \infty} \inf_z\frac{K_2(z,x) }{\overline{F}(x)} = 1,
$$

Thus, \eqref{Gnew1} holds for  $n=1$.  By the same argument we get that $\overline{G_\rho} (z,x)$ is long tailed as $x\to \infty$, uniformly in $z$.

Consider the second convolution $\overline{G^{* 2}_\rho} (z,x)$.
 By the definition of the convolution given in Theorem~\ref{t-ren} we have
\begin{align*}
\overline{G^{* 2}_\rho} (z,x)&= \left( \int_{ w< -\phi(x) }+ \int_{-\phi(x)\leq  w \leq \phi(x)}+ \int_{\phi(x)<  w \leq x-\phi(x)} + \int_{ w >x-\phi(x)}\right)
 \overline{G}_\rho(z- w,x-  w)G_\rho(z,d w)  \\
&:= K_{21}( z,x )+ K_{22}( z,x)+ K_{23}( z,x)+K_{24}( z,x ).
\end{align*}
Similarly to the argument for $K_{1}(z,x)$,  we get $\sup_z K_{21} (z,x)= o (\overline{F}(x))$ as $x^0\to \infty$.

 From the case when $n=1$ we know that    there exist $ 0<C_5<C_6<\infty$ such that
\begin{equation}\label{tail}
 C_5 \leq \liminf_{x^0 \to \infty} \frac{\overline{G}_\rho(z,x)}{\overline{F}(x)}\leq
\limsup_{x^0 \to \infty} \frac{\overline{G}_\rho(z,x)}{\overline{F}(x)}< C_6,
\end{equation}
uniformly in $z$.
 This allows to bound
$$
K_{23}(z,x)\leq  C_7 \int_{\phi(x)<x_1\leq x-\phi(x)} \overline{F}(x- w) F(d w),
$$
which is $o(\overline{F}(x))$ as $x^0\to \infty$ (see \cite[Thm. 3.7]{FKZ13} for the one-dimensional case,
the argument in the multi-dimensional one is the same).  By the same argument as for $K_2(z,x)$,
we  conclude that $K_{22}(z,x)= \overline{F}(x)(1+o(1))$, $x^0\to \infty$.    Finally,
 by $\phi$-insensitivity of $\overline{F}$, Remark \ref{phi10} and \eqref{tail} we have
$$
K_{24}(z,x)\leq \int_{x-\phi(x)<x_1} G_\rho(z,d w) = \overline{G_\rho}(z,x-\phi(x))=  \overline{F}(x) (1+o(1)),
$$
\begin{align*}
K_{24}(z,x)& \geq \int_{x_1\geq x+ \phi(x)} \overline{G}_\rho(z- w,x-  w)G_\rho(z,d w) \geq
\inf_y \overline{G}_\rho(y,-\phi(x)) \overline{G}_\rho(z,x+\phi(x))\\
&=
 \overline{F}(x) (1+o(1)).
\end{align*}

 Thus,  $K_{24}(z,x)= \overline{F}(x)(1+o(1))$.
For general $n$ the proof follows by induction  and the argument similar to that for  $n=2$.

  To prove Kesten's bound \eqref{Gnew2} we follow again \cite[Ch. 3.10]{FKZ13} and \cite[p. 5439]{Om06}.
Note that
$$
\overline{G^{* n}_\rho}(   z,x)\leq \sum_{i=1}^d \overline{G^{* n}_{\rho,i}}(   z, x),
$$
where $G^{* n}_{\rho,i}(z, x):= G^{* n}_{\rho}(z, \real \times \ldots \times (-\infty, x_i)\times \ldots \times \real)$ are marginals of $G^{* n}_{\rho}$.
Now generalizing \cite[Ch. 3.10]{FKZ13}
 to our set-up of $G^{* n}_{\rho,i}$
we can conclude that for each $\epsilon >0$ there exists
constant $ C_7$ such that
\[\overline{G^{* n}_\rho}(   z,x)\leq  C_7(1+\epsilon)^n \sum_{i=1}^d \overline{G}_{\rho,i}(   z, x),\]
implying
\[\overline{G^{* n}_\rho}(   z,x)\leq  C_7 d(1+\epsilon)^n \overline{G}_{\rho}(   z, x)\]
and we can use  \eqref{Gnew1}
to conclude \eqref{Gnew2}.

\end{proof}

\begin{proof}[Proof of Theorem~\ref{p2}]
\emph{1. Case  $B\in [0,\infty)$.}
Let
$$
\mathcal{G}(x,\cdot):= (1-\rho)\sum_{k=0}^\infty \rho^k G_\rho^{* k}(x, \cdot).
$$
Applying \eqref{Gnew2}  with $\epsilon< \frac{1-\rho}{\rho}$  we can pass to the limit
\begin{equation}\label{eqsecond}
\lim_{z^0\to \infty}\frac{\overline{\mathcal{G}}\big(x, z\big)}{\overline{F}\big(z\big)} =(1-\rho)\sum_{k=1}^\infty k \rho^k = \frac{\rho}{1-\rho}.
\end{equation}

We prove that
\begin{equation}\label{Rlim}
\begin{split}
 \lim_{x^0\to \infty}  \frac{h(x)}{\ell(x)} =  \lim_{x^0\to \infty} \frac{\int_0^\infty  e^{-\lambda r} \Pp(T>r) \Ee^{x} \ell(Y_r)dr}{\ell(x)}=  \lim_{x^0\to \infty} \frac{\int_\rd  \ell(x+w) q(x,w+x)dw }{\ell(x)}= \rho.
\end{split}
\end{equation}
We  use \eqref{q22} and that  $\ell \in \Bb_b^+(\rd)$ and  is long tailed. Indeed, by the same idea as that used in the proof of  Lemma~\ref{GF2},  we split
\begin{align*}
\int_\rd  \frac{\ell(x+w)q(x,w+x)}{\ell(x)} dw &= \left(\int_{|w|\leq |\phi(x)| } + \int_{|w|>|\phi(x)|} \right) \frac{\ell(x+w)q(x,w+x)}{\ell(x)} dw \\
&:= I_1(x)+ I_2(x),
\end{align*}
where the function $\phi(x)=(\phi_1(x),\dots,\phi_d(x))$, $\phi_i(x)>0$,  is such that  $\ell$ is $\phi$-insensitive.

For any $\epsilon=\epsilon (b^0)>0$ and large enough $x^0\geq b^0$ we get
$$
I_1(x) \leq (1+ \epsilon (b^0)) \int_{|w|\leq |\phi(x)| } q(x,w+x) dw
\leq (1+ \epsilon (b^0))\rho
$$
and similarly
$$
I_1(x) \geq (1-\epsilon (b^0)) \rho.
$$
Thus,  $\lim_{x^0\to \infty} I_1(x)=\rho$.
By \eqref{lp}  we get
$$
I_2(x) \leq C  \int_{|w|\geq |\phi(x)|} \frac{q(x,w+x)}{\ell(x)} dw \leq\frac{ Ce^{-c|\phi(x)|}}{\ell(x)}  \to 0 \quad \text{as $x^0\to \infty.$}
$$

Now we investigate the asymptotic behaviour of $\int_\rd h(x-y) \mathcal{G}(z,dy)$ (at the moment we assume that $z\in \rd$ is fixed; as we will see, it does not affect the asymptotic behaviour of the convolution). From now $\phi$ is such   that  both  $\ell$ and $\overline{F}$ and $\phi$-insensitive.
 Split the integral:
\begin{align*}
\int_\rd h(x-y) \mathcal{G}(z,dy)&= \left(\int_{y\leq -\phi(x)} + \int_{-\phi(x) \leq y\leq \phi(x)} + \int_{\phi(x)<y<x-\phi(x)}\right.\\
&\qquad \left.+ \int_{x-\phi(x)}^{x+\phi(x)}+ \int_{x+\phi(x)}^\infty \right) h(x-y) \mathcal{G}(z,dy)
\\
&:= J_1(z,x)+ J_2(z,x)+ J_3(z,x)+ J_4(z,x) + J_5(z,x).
\end{align*}
Observe that $B\in [0,\infty)$ implies that $\ell(x)$ is either comparable with the monotone function $\overline{F}(x)$, or $\ell(x)= o(\overline{F}(x))$ as $x^0\to \infty$.
By \eqref{Rlim}, this allows to estimate $J_1$ as
\begin{align*}
J_1(z,x)&\leq \sup_{w\geq \phi(x)} h(x+w)   \mathcal{G}(z,(-\infty,-\phi(x)])\leq C_1 \ell(x) \mathcal{G}(z, (-\infty,-\phi(x)]) \\
&= o(\ell(x))
= o (\overline{F}(x)), \quad x^0\to \infty,
\end{align*}
uniformly in $z$.  From \eqref{Rlim}  we have
\begin{equation}\label{J2}
J_2(z,x)= \rho \ell(x) (1+ o(1)), \quad x^0\to \infty,
\end{equation}
 uniformly in $z$.  Let us estimate $J_3(z,x)$.  Under the assumption $B\in [0,\infty)$ we have
\begin{equation}\label{J3}
J_3(z,x) \leq C_2 \int_{\phi(x)<y<x-\phi(x)}     \overline{F} (x-y) F(dy).
\end{equation}
Since $F$ is subexponential, the right-hand side of \eqref{J3}  is $o(\overline{F}(x))$ as $x^0\to \infty$; in the one-dimensional case this is stated in \cite[Thm. 3.7]{FKZ13}, the proof in the multi-dimensional case is literally the same.

For $J_4$ we have
\begin{equation}\label{J4}
\begin{split}
J_4(z,x) &\leq C_3 \left( F(x+\phi(x)) -F(x-\phi(x))\right) = C_3 \left( \overline{F}(x-\phi(x))- \overline{F}(x+\phi(x))\right)\\
&\leq   o (\overline{F}(x)),
\end{split}
\end{equation}
uniformly in $z$. Finally, for $J_5$ we have

\begin{align*}
J_5(z,x)& \leq  C_4 \sup_{w\leq -\phi(x)} h(w)    \overline{F}(x)= o(\overline{F}(x)).
\end{align*}
Thus, in the case $B\in [0,\infty)$ we get  the first and the second relations in \eqref{ub}.

2. \emph{Case  $B =\infty$.}  The argument for $J_1$ and $J_2$  remains the same. For $J_3$ we have
\begin{align*}
J_3(z,x) &\leq \ell(\phi(x)) (\overline{F}(\phi(x))- \overline{F}(x-\phi(x))) \leq C_5 \ell(\phi(x)) \overline{F}(\phi(x))\\
&\leq C_5 \ell^2(\phi(x)).
\end{align*}
By Remark~\ref{phi10} we can chose $\phi$ such that  $|\phi(x)|\asymp  |x| \ln^{-2} |x|$ as $x^0 \to \infty$.
Since in the case when $B=\infty$ the function $\ell$ is assumed to be regularly varying, it has a power decay, $J_3(z,x)= o (\ell(x))$, $x^0\to \infty$. By the same argument,
 $J_i(z,x)= o(\ell(x))$, $i=4,5$,  which proves the last relation in \eqref{ub}.

\qedhere
\end{proof}

In the next section we provide the examples in which \eqref{q22} is satisfied.
\begin{remark}
In the case when $Y$ is degenerate, e.g. $Y_t=x+at$, one can derive the asymptotic behaviour of $u(x)$ by a much more simple procedure.  For example, let
$d=1$, $T\sim \Exp(\mu)$,  $\mu>0$,  $Y_t=at$ with $a>0$, and $\ell(x)= \overline{F}(x)$, $x\geq 0$, and $\ell(x)=0$ for $x<0$.
This special type of the  function $\ell$ appears in the multivariate ruin problem, see also \eqref{GS2} below.
In this case $\rho=\frac{\lambda}{\lambda+\mu}$.  Then
$$
 \overline{G}(z) = \int_0^\infty \lambda e^{-(\lambda +\mu)t} \overline{F}(z+at)dt. 
 $$
Direct calculation gives 
$ \overline{G}(z)=
  \overline{F}(x)(1+ o(1))$ as $x^0\to \infty$, implying that
  $$
u(x)= \frac{\lambda}{\mu} \overline{F}(x)(1+o(1)), \quad x\to \infty.
$$
\end{remark}

\section{Examples}\label{sec:examples}

We begin with a simple example which illustrates Theorem~\ref{p2}.   Note that in the L\'evy case $\mathfrak{p}_s(x,w)$ depends on the difference $w-x$; in order to simplify the notation we  write in this case $ \mathfrak{p}_s(x,w) = p_s(w-x)$,
$$
q(w):=\int_0^\infty  \lambda e^{-\lambda s} \Pp(T>s)  p_s (w) ds
$$
 and \begin{equation}\label{q2}
G(dz) =
 \int_\rd F(dz+w) q(w)dw.
\end{equation}
We prove below a technical lemma, which provides the necessary estimate for $\mathfrak{p}_s(x,w)$ in the case when

\medskip

 a)  $Y_t = x +at  + Z^{\rm small}_t$, where $a\in \rd$  and $Z^{\rm small}$ is a L\'evy process with jumps size smaller than $\delta$, i.e. its characteristic exponent is of the form
\begin{equation}\label{psi1}
\psi^{\rm small}(\xi) := \int_{|u|\leq \delta} (1-e^{i\xi u} +  i \xi u) \nu(du),
\end{equation}
where $\nu$ is a L\'evy measure;

\medskip

b) $Y_t= x + at +V_t$, where $V_t$ is an Ornstein-Uhlenbeck process driven by $Z^{\rm small}_t$, i.e. $V_t$ satisfies the SDE
$$
dV_t = \vartheta V_t dt + d Z_t^{\rm small}.
$$
We assume that $\vartheta <0$ and  that $Z_t^{\rm small}$ in this model has only positive jumps.

Assume that for some $\alpha\in (0,2)$ and $c>0$
\begin{equation}\label{dens1}
\inf_{\ell,\jmath\in \mathbb{S}^d}  \int_{\ell   \cdot u>0} \big(1-\cos (R \jmath   \cdot u)\big) \nu(du)\geq c R^\alpha, \quad R\geq 1,
\end{equation}
where $\mathbb{S}^d$ is the sphere in $\rd$. Under this condition there exists (cf. \cite{K11}) the  transition probability density of $Y_t$ in both cases. Let
\begin{equation}
k_t(x):= \Ff\big( e^{- \psi_t(\cdot)}\big)(x),
\end{equation}
$$
\psi_t(\xi)= -it a \cdot\xi  +\int_0^t  \psi^{\rm small}( f(t,s)\xi)ds,
$$
where $f(t,s)= \I_{s\leq t} $ in the case a), and $f(t,s)= e^{(t-s)\vartheta}  \I_{0\leq s\leq t}  $
in case b). Note that since $\vartheta<0$  we have  $0<f(t,s)\leq 1$.
Note that in the case b) $\mathfrak{p}_t(0,x)=k_t(x)$.

Note that we always have
\begin{equation}\label{kt}
k_t(x) \leq (2\pi)^{-d/2} \int_\rd e^{- \int_0^t \mathrm{Re}  \psi^{\rm small}( f(t,s)\xi)ds } d\xi\leq  (2\pi)^{-d/2} \int_\rd e^{- c |\xi|^\alpha \int_0^t  |f(t,s)|^\alpha ds} d\xi,
\end{equation}
where in the second inequality we used \eqref{dens1}.
\begin{lemma}\label{dens-est}
 Suppose that \eqref{dens1} is satisfied. We have
\begin{equation}\label{pt1}
k_t(x)\leq
 \begin{cases}
  C e^{ - (1-\epsilon) \theta_\nu   |x-at| } &  \text{if} \quad  t>0, \, |x-at|\gg t\vee 1,\\
 C t^{-d/\alpha} & \text{if} \quad  t>0, \, x\in \rd,
 \end{cases}
\end{equation}
in case a), and

\begin{equation}\label{pt2}
k_t(x) \leq
\begin{cases}
C e^{- (1-\epsilon) \theta_\nu  |x-at|}, &  \text{if} \quad   t>0, \, x\in \rd,\,  |x-at|\gg 1,\\
C, & t>0, \quad  x\in \rd.
\end{cases}
\end{equation}
in case b). Here $\theta_\nu>0$ is a constant depending on the support of $\nu$ and $\epsilon>0$ is arbitrary small.
\end{lemma}
\begin{proof}  For simplicity, we assume that in case b) we have $\vartheta=-1$.

Without loss of generality assume that $x>0$.
Rewrite  $\mathfrak{p}_t(x)$ as
$$
k_t(x)= (2\pi)^{-d} \int_\rd e^{H(t,x,\xi)}d\xi,
$$
where
$$
H(t,x,\xi)= i\xi (x-at)- \psi_t(-\xi).
$$
It was shown in Knopova \cite[p. 38]{K11}, that the function $\xi\mapsto H(t,x,i\xi)$, $\xi\in \rd$,  is convex,   there exists  a solution to
$\nabla_\xi  H(t,x,i\xi)=0$, which we denote by $\xi=\xi(t,x)$, and  by non-degeneracy condition \eqref{dens1} we  have $x\cdot \xi>0$,   $|\xi(t,x)|\to\infty$, $|x|\to \infty$.  Further, by the same way as in \cite{K11}, see also Knopova, Schilling  \cite{KS12} and Knopova, Kulik \cite{KK11} (for the one-dimensional version), one can apply the Cauchy-Poincar\'e theorem and  get
\begin{equation}\label{CP1}
\begin{split}
k_t(x)&= (2\pi)^{-d}   \int_{i\xi (t,x)+ \rd} e^{H(t,x,z)} dz\\
&=  (2\pi)^{-d}   \int_{ \rd} e^{H(t,x,i\xi(t,x)+ \eta)} d\eta\\
&= (2\pi)^{-d}   \int_{ \rd} e^{\Re H(t,x,i\xi(t,x)+ \eta)} \cos \big( \Im H(t,x,i\xi(t,x)+ \eta)\big)d\eta\\
&\leq (2\pi)^{-d}   \int_{ \rd} e^{\Re H(t,x,i\xi(t,x)+ \eta)}d\eta.
\end{split}
\end{equation}
We have
\begin{align*}
\Re H(t,x,i\xi+ \eta) &= H(t,x,i\xi) -\int_0^t\int_{|u|\leq \delta} e^{f(t,s)\xi \cdot  u}  \big( 1-\cos( f(t,s) \eta \cdot u) \big)\nu(du)\,ds\\
&\leq H(t,x,i\xi) -\int_0^t \int_{|u|\leq \delta,\,\xi \cdot u>0}  \big( 1-\cos( f(t,s) \eta \cdot u) \big)\nu(du)\,ds\\
&\leq H(t,x,i\xi) - c |\eta|^\alpha\int_0^t |f(t,s)|^\alpha ds,
\end{align*}
where
$$
H(t,x,i\xi) = -(x-at)\cdot \xi + \int_0^t \int_{|u|\leq \delta} \Big( e^{f(t,s) \xi\cdot u} -1-f((t,s)\xi \cdot u) \Big) \nu(du) ds,
$$
and in the last inequality we used \eqref{dens1}.
Hence,
\begin{equation}\label{diag10}
k_t(x)\leq (2\pi)^{-d} e^{H(t,x,i\xi)} \int_\rd e^{- c |\eta|^\alpha\int_0^t |f(t,s)|^\alpha ds}d\eta.
\end{equation}
Now we estimate the function  $H(t,x,i\xi)$. Differentiating, we get
\begin{align*}
\partial_\xi H(t,x,i\xi)& = -(x-at)\cdot e_\xi + \int_0^t\int_{|u|\leq \delta} \Big( e^{f(t,s) \xi\cdot u} -1\Big)f(t,s) u \cdot e_\xi  \nu(du) ds\\
&=: -(x-at)\cdot e_\xi + I(t,x,\xi),
\end{align*}
where $e_\xi = \xi /|\xi|$.
For large $|\xi|$ we can estimate $I(t,x,\xi)$ as follows:
\begin{align*}
I(t,x,\xi)&\leq C_1 \int_0^t \int_{|u|\leq \delta} |f(t,s) u|^2 e^{f(t,s) \xi \cdot u} \nu(du)\,ds\\
&\leq C_1 e^{\delta |\xi| \max_{s\in[0,t]}f(t,s)} \int_0^t f^2(t,s)  ds
\end{align*}
for some constants $C_1$.
For the lower bound we get for
\begin{align*}
I(t,x,\xi)&\geq C_2 \int^t_{(1-\epsilon_0)t} \int_{ |u|\leq \delta, \, \xi \cdot u > |\xi|(\delta-\epsilon)} |f(t,s) u|^2 e^{f(t,s) \xi \cdot u} \nu(du)\,ds\\
&\geq C_3 e^{(\delta-\epsilon) |\xi| \min_{s\in[(1-\epsilon_0)t,t]}f(t,s)} \int_{(1-\epsilon_0)t}^t  f^2(t,s)  ds,
\end{align*}
where $C_2,C_3>0$ are some constant,
$\epsilon_0,\epsilon\in (0,1)$.
 Thus, we get
\begin{equation}\label{I1}
C_3 t e^{ (\delta-\epsilon) |\xi|} \leq I(t,x,\xi)\leq  C_1 t e^{\delta |\xi|}
\end{equation}
in case a), and
\begin{equation}\label{I2}
C_3 e^{ (\delta-\epsilon) e^{\epsilon_0 }   |\xi|} \leq I(t,x,\xi)\leq  C_1   e^{\delta |\xi|}
\end{equation}
in case b).

In particular, this estimate implies that there exists $c_0>0$ such that $(x-at)\cdot e_\xi\geq c_0$, i.e., $e_\xi$ is directed towards $x-at$, in particular, cannot be orthogonal to $x-at$.

From now we treat each case separately.

\emph{Case a)}.
If  $|x-at|/t\to \infty$,  we get for any $\zeta\in (0,1)$
$$
(1-\zeta)\theta_\nu  \ln \big( |x-at|/t\big) (1+o(1))\leq  \xi(t,x)\leq (1+\zeta) \theta_\nu  \ln \big( |x-at|/t\big) (1+o(1)),
$$
where $\theta_\nu$ is the constant which depends on the support $\supp \nu$.
Therefore,
\begin{equation}\label{HLevy}
H(t,x,i\xi(t,x))\leq - (1-\zeta) \theta_\nu  |x-at|\ln \big( |x-at|/t\big) + C_4,
\end{equation}
for $ t>0$, $|x-at|\gg t$ and some constant $C_4$.

It remains to estimate the integral term  in \eqref{diag10}.
We have
$\int_0^t f^\alpha (t,s) ds= t $, hence
\begin{equation}\label{diag1}
\int_\rd e^{- c |\eta|^\alpha\int_0^t  f^\alpha (t,s) ds}d\eta=  C_5 t^{-d/\alpha}.
\end{equation}
Thus, we get
\begin{equation}\label{offdiag1}
k_t(x)\leq C_6 t^{-d/\alpha} e^{ - (1-\zeta) \theta_\nu |x-at| \ln \big(|x-at|/t\big)}.
\end{equation}
For $t\geq 1$ the first estimate in \eqref{pt1} follows from \eqref{offdiag1}, because $t^{-d/\alpha}\leq 1$.

 Consider now  the case $t\in (0,1]$.
For $t\in (0,1]$ and  $|x|\gg 1$  we have  for $K$ big enough  and some constant $C_7$
\begin{align*}
 e^{ - \zeta(1-\zeta) \theta_\nu |x-at| \ln \big(|x-at|/t\big)}&\leq  e^{ - \zeta (1-\zeta) \theta_\nu (|x|-|a|)| \ln \big(|x-at|/t\big)}\leq C_7 e^{-K \ln \big(|x-at|/t\big)}.
\end{align*}
Without loss of generality, assume that   $K>d/\alpha$. Then
\begin{align*}
k_t(x)&\leq C_8 t^{-d/\alpha} e^{ - (1-\zeta)^2 \theta_\nu |x-at| \ln \big(|x-at|/t\big)- \zeta(1-\zeta) \theta_\nu |x-at| \ln \big(|x-at|/t\big)}\\
&\leq C_9 t^{-d/\alpha} \left(\frac{t}{|x-at|}\right)^K e^{ - (1-\zeta)^2 \theta_\nu|x-at| }  \\
&\leq C_{10} e^{- (1-\zeta)^3 \theta_\nu|x-at| },
\end{align*}
which proves the first estimate  in \eqref{pt1} by taking
$1-\epsilon =(1-\zeta)^3$.
For the third estimate in \eqref{pt1} observe that $H(t,x,i\xi)\leq 0$. Then the bound follows from  \eqref{diag1}.

\emph{Case b).} If $|x-at|\to \infty$ we get for any $\zeta\in (0,1)$
$$
(1-\zeta)\theta_\nu  e^{ \epsilon_0} \ln |x-at| (1+o(1))\leq  \xi(t,x)\leq (1+\zeta) \theta_\nu   \ln |x-at| (1+o(1)),
$$
where $\theta_\nu$ is the constant which depends on the support $\supp \nu$.

Now we estimate the right-hand side in \eqref{diag1} in case  b). Since   $\int_0^t  f^\alpha(t,s)  ds= \alpha^{-1}  (1- e^{-\alpha t}) $, we get
\begin{equation}\label{diag2}
\int_\rd e^{- c |\eta|^\alpha\int_0^t  f^\alpha(t,s) ds}d\eta\leq C_{11}
\end{equation}
 for some constant $C_{11}$.
Thus,  there exist $C_{12}>0$ and $\epsilon\in (0,1)$ such that for $x\in \rd$ and $t>0$ satisfying  $|x-at|\gg 1 $ we get
\begin{align*}
k_t(x)
&\leq C_{12} e^{-(1-\epsilon)  \theta_\nu|x-at| },
\end{align*}
which proves \eqref{pt2} for large $|x-at|$. Finally, boundedness of $\kappa_t(x)$ follows from \eqref{kt} and that in case b) we have $c\leq \int_0^t f^\alpha(t,s) ds \leq C$ for all $t>0$.
\end{proof}
\begin{remark}\rm The same  estimates  can be shown also for the model $Y_t =  x+ at+  \sigma B_t +  Z^{\rm small}_t$.

b) Note that $\epsilon>0$\,  in the exponent in \eqref{pt1} and \eqref{pt2} can be chosen arbitrary close to  $0$, i.e. is in a sense sharp.
\end{remark}
\begin{lemma}\label{qest}
Let $Y$ be as in case a). There exist $C>0$ and $\epsilon\in (0,1)$ such that  the  estimate
$$
q(x)\leq C  e^{-(1-\epsilon)\theta_q |x|}, \quad |x|\gg 1,
$$
holds true, where
\begin{equation}\label{thetaq}\theta_q= \theta_\nu \wedge \lambda/ (2|a|).
\end{equation}
\end{lemma}
\begin{proof}
We use Lemma~\ref{dens-est}. 
We have
\begin{align*}
q(x)=  \left(\int_{\{t: |x|>|a|t + (t\vee 1)\}} + \int_{\{t: |x|\leq |a|t + (t\vee 1)\}} \right) e^{-\lambda t} \kappa_t(x) 
 dt
:= I_1 + I_2.
\end{align*}
For $I_1$ we use the triangle inequality
\begin{align*}
I_1 &\leq C_1 e^{-(1-\epsilon)\theta_\nu |x| }  \int_{\{t: |x|>|a|t\}} e^{-(\lambda- (1-\epsilon) \theta_\nu |a|) t}  dt \\
 &\leq C_1 e^{-(1-\epsilon)\theta_\nu |x| }
  \begin{cases}
  C_2 & \text{if} \quad  \lambda> (1-\epsilon) \theta_\nu |a|,\\
  C_2 e^{\frac{(1-\epsilon) \theta_\nu |a|-\lambda}{ |a|} |x|}   & \text{if} \quad  \lambda<  (1-\epsilon) \theta_\nu |a|,
  \end{cases}
  \end{align*}
  where $C_1,C_2>0$ are certain constants and we exclude the equality case by choosing appropriate $\epsilon>0$.   Hence,
  $$
  I_1\leq
  \begin{cases}
  C_3e^{-(1-\epsilon)\theta_\nu |x| }, & \lambda> (1-\epsilon) \theta_\nu |a|,\\
  C_3 e^{- \frac{\lambda}{|a|} |x|},   & \lambda<  (1-\epsilon) \theta_\nu |a|
  \end{cases}
$$
for some $C_3>0$.

For $I_2$ we get, since $|x|\gg 1$,
\begin{align*}
I_2 &\leq C_4 \int_{\{ t: \, t >  |x|/(2|a|) \} }   t^{-d/\alpha}  \lambda e^{-\lambda t}dt \leq C_5 e^{-\frac{(1-\epsilon) \lambda}{ 2|a|} |x|}.
\end{align*}
Thus,  there exists $\epsilon>0$ and $C>0$ such that
$$
I_k \leq C e^{- (1-\epsilon)(\theta_\nu \wedge \lambda/(2|a|)) |x|}, \quad k=1,2.
$$
This completes the proof.
\end{proof}

Consider now the estimate in case b). Recall that we assumed that  the process $Y$ has only positive jumps. This means, in particular, that in the transition probability density $\mathfrak{p}_t(x,y)$ we only have $y\geq x$ (in the coordinate sense). Under this assumption it is possible to show that
$q(x,y)$ (cf. \eqref{q1}) decays exponentially fast as $|y-x|\to \infty$.
\begin{lemma}\label{qest2}
 In the case b) there exist $C>0$ and $\epsilon\in (0,1)$ such that
$$
q(x,y)\leq C  e^{-(1-\epsilon)\theta_q |y-x|}, \quad |y-x|\gg 1,
$$
where $\theta_q$ is the same as in Lemma~\ref{qest}.
\end{lemma}
\begin{proof}
From the representation  $Y_t = e^{-t}(x+ \int_0^t e^s dZ_s^{\rm small})$ and \eqref{pt2} we get
$$
\mathfrak{p}_t(x,y)\leq
C e^{- (1-\epsilon) \theta_\nu  |y-xe^{-t}-at|},\quad  t>0, \, x,y>0,  |y-xe^{-t}-at|\gg 1.
$$
Similarly to the proof of Lemma~\ref{qest},  we have
 \begin{align*}
q(x,y)&\leq C_1 \int_{\{t: |y-x|> |a| t\}} e^{-\lambda t} e^{-(1-\epsilon)\theta_\nu |y-e^{-t}x-at|} dt + C_2\int_{\{t: |y-x|\leq  |a| t\}}  e^{-\lambda t} dt\\
&:= I_1 + I_2.
\end{align*}
Since $y>x$, we have $|y-e^{-t}x|= y- e^{-t}x> y-x>0$ and therefore
\begin{align*}
M_1 & \leq C_1 \int_{\{t: |y-x|> |a| t\}} e^{-(\lambda- (1-\epsilon) \theta_\nu |y-e^{-t}x|- (1-\epsilon) \theta_\nu |a|) t}  dt \\
&\leq C_1 e^{-(1-\epsilon)\theta_\nu |y-x| }  \int_{\{t: |y-x|> |a| t\}} e^{-(\lambda- (1-\epsilon) \theta_\nu |a|) t}  dt \\
 &\leq C_1 e^{-(1-\epsilon)\theta_\nu |y-x| }
  \begin{cases}
  C_3, & \lambda> (1-\epsilon) \theta_\nu |a|,\\
  C_3 e^{\frac{(1-\epsilon) \theta_\nu |a|-\lambda}{ |a|} |y-x|},   & \lambda<  (1-\epsilon) \theta_\nu |a|.
  \end{cases}
  \end{align*}
Hence,
  $$
  I_1\leq
  \begin{cases}
  Ce^{-(1-\epsilon)\theta_\nu |y-x| } & \text{if} \quad
\lambda> (1-\epsilon) \theta_\nu |a|,\\
  C e^{-  \frac{\lambda}{|a|}  |y-x|}   & \text{if} \quad \lambda<  (1-\epsilon) \theta_\nu |a|.
  \end{cases}
$$
Clearly,
\begin{align*}
I_2 &\leq  C  e^{-\frac{(1-\epsilon) \lambda}{|a|} |y-x|},
\end{align*}
which completes the proof.
\end{proof}

\begin{remark}\rm
Direct calculation shows  that estimate \eqref{q22} is not satisfied for an Ornstein-Uhlenbeck process driven by a Brownian motion, unless $\lambda>\theta$.
\end{remark}

Consider an example in $\real^2$, which illustrates how one can get the asymptotic of $u(x)$ along curves.
\begin{example}\label{example2d}\rm
Let $d=2$ and $x=(x_1(t), x_2(t))$. We assume that   $x_i=x_i(t)\to \infty$ as $t\to \infty$ in such a way that $x(t)\in \mathbb{R}^2\backslash \partial$.    Suppose that  $F\in WS(\mathds{R}^2_+)$ and factorizes as $F(x)= F_1(x_1) F_2(x_2)$.
Assume that the assumptions of Theorem~\ref{p2} are satisfied with $B\in (0,\infty)$.
Since
$$
\overline{F}(x)=1- F_1(x_1)F_2(x_2)= \overline{F}_1(x_1) F(x_2) + \overline{F}_2(x_2),
$$
we get in the case of Theorem~\ref{p2}  and $B\in (0,\infty)$ that
$$
u\big(x\big) = \frac{B \rho }{1-\rho} \overline{F} (x)(1+ o(1))=   \frac{B \rho }{1-\rho}\Big(\overline{F_1}\big(x_1(t)\big)+ \overline{F_2}\big(x_2(t)\big)\Big) (1+ o(1)) \quad \text{as $t\to \infty.$}
$$
Thus,  taking different (admissible)  $x_i(t)$, $i=1,2$, we can achieve different effects in the asymptotic of $u(x)$. For example, assume that for $z\geq 1$
$$
\overline{F}_i(z)=c_i z^{-1-\alpha_i}, \quad i=1,2,
$$
where  $c_i, \alpha_i>0$, are suitable constants.
Direct calculation  shows that $F_i(x)$ are subexponential and the relations in \eqref{subexp} hold true. Note that the behaviour of $F$  depends on the constants $\alpha_i$ and on the coordinates of $x$. We have:
\begin{equation}
\overline{F}(x(t))=
 \begin{cases}
 \frac{c_1(1+o(1))}{(x_1(t))^{1+\alpha_1}} & \text{if} \quad \lim_{t\to\infty} \frac{x_1^{1+\alpha_1}(t)}{x_2^{1+\alpha_2}(t)}=0,\\
 \frac{c_2(1+o(1))}{(x_2(t))^{1+\alpha_2}} & \text{if} \quad \lim_{t\to\infty} \frac{x_1^{1+\alpha_1}(t)}{x_2^{1+\alpha_2}(t)}=\infty,\\
(1+o(1))\Big( \frac{c_1}{x_1(t)^{1+\alpha_1}}+\frac{c_2 }{x_2(t)^{1+\alpha_2}}\Big) & \text{if} \quad \lim_{t\to\infty} \frac{x_1^{1+\alpha_1}(t)}{x_2^{1+\alpha_2}(t)}=c\in (0,\infty).\\
 \end{cases}
 \end{equation}
Taking, for example,
 $x=(t,t)$ or $x=(t,t^2)$ we get the behaviour of  $u(x)$ along the line $y=x$, or along the parabola $y=x^2$, respectively.
\end{example}

\begin{example}\label{example2db}\rm
Let $d=2$ and suppose that the   generic jump is of the form $U=(\varrho \Xi, (1-\varrho)\Xi)$, where  $\varrho \in (0,1)$ and the distribution function $H$ of the random variable $\Xi$ is subexponential on $[0,\infty)$.
Then
$\overline F (x)= \overline{H}\left(\frac{x_1}{\varrho}\wedge \frac{x_2}{1-\varrho}\right)$,
$F\in WS (\mathds{R}_+^2)$, and

$$
\overline{F}(x(t))=
\left\{\begin{array}{lr}
\overline{H}\left(\frac{x_1(t)}{\varrho}\right)(1+o(1))& \text{if} \quad  \lim_{t\rightarrow\infty} \frac{x_1(t)(1-\rho)}{x_2(t)\varrho}\leq 1\\
\overline{H}\left(\frac{x_2(t)}{1-\varrho}\right)(1+o(1))& \text{if} \quad  \lim_{t\rightarrow \infty} \frac{x_1(t)(1-\rho)}{x_2(t)\varrho}> 1.
\end{array}
\right.
$$
Thus, one can get the asymptotic behaviour of $u(x)$ provided that the assumptions of Theorem~\ref{p2} are satisfied with $B\in (0,\infty)$.
\end{example}

\begin{example}\rm
Let $x\in \rd$,  $T\sim \Exp(\mu)$, independent of $X$, and $Y$ is as in the cases a)  or b). Recall that in this case $\rho=\frac{\lambda}{\lambda+ \mu}$.  Let $\ell(x) = \I_{|x|\leq r}$.
Then
$$
u(x)= \int_0^\infty \Pp^x ( |X^\sharp_t|\leq r)dt= \int_0^\infty \mu e^{-\mu t} \Pp^x (|X_t|\leq r)dt.
$$
Then the assumptions of Theorem \ref{p2} are satisfied with $B=0$;
therefore,
$$
u(x) = o(1)\overline{F}(x)  \quad \text{as $x^0 \to \infty$.} 
$$

If $\ell(x)= \I_{\min x_i\geq r}$
then
$$
u(x)= \int_0^\infty \Pp^x ( X^\sharp_t\geq r)dt= \int_0^\infty \mu e^{-\mu t} \Pp^x (\min_{1\leq i\leq d} X_t^i \geq r)dt.
$$
Then we are in the situation of Theorem~\ref{p2}   with $B=\infty$, hence,
$$
u(x)=\frac{\lambda}{\mu}  (1+o(1)), \quad \text{as $x^0 \to \infty.$}
$$
\end{example}

\medskip

\begin{example}\label{risk}\rm
At the end of this section we consider a simple example when $T$ is not independent of $X$.
We consider a simple well-known one-dimension case $X_t = x+ at - Z_t$ with  $a>0$,
$\Ee U_1=\mu$, $N_t \sim \Pois(\lambda)$
and $T=\inf\{t\geq 0: X_t<0\}$ being a ruin time.
We put
\begin{equation}\label{ruinell}
\ell(x)=  \lambda \overline{F} (x).
\end{equation}
   Then the renewal equation \eqref{ren1} for  $u(x)$ is
\begin{equation}\label{u1b}
u(x)= \int_0^\infty \lambda e^{-\lambda t} \overline{F} (x+at)dt    + \int_0^\infty \lambda e^{-\lambda t} \int_0^{x+at} u(x+at-y) F(dy) \, dt.
\end{equation}
Changing the variables we get
\begin{align*}
u(x)= h(x) + \int_{-\infty}^x u(x-z) G(dz)
\end{align*}
with
$h(x)= \int_0^\infty \lambda e^{-\lambda t} \overline{F} (x+at)dt $ and
$$
G(dz)=\I_{z\geq 0} \int_0^\infty \lambda e^{-\lambda t}  F(dz+ at)  \, dt  + \I_{z<0}  \int_{-z/a}^\infty \lambda e^{-\lambda t}  F(dz+ at)  \, dt.
$$
Note that $\supp G = \real$, and $G(\real)=1$, hence, the result of Theorem~\ref{p2} cannot be applied directly.
In this situation the well-known approach is more suitable; below  we recall this approach.

Taking
\begin{equation}\label{vudef}
v(x)= 1- u(x)
\end{equation}
and starting from \eqref{u1b}
we end up with
\begin{align*}
&v(x)=
-\int_0^\infty \lambda e^{-\lambda t} \overline{F} (x+at)dt \\&\qquad\qquad   +
\int_0^\infty \lambda e^{-\lambda t} \left(\int_0^{x+at}F(dy) +\overline{F} (x+at)-
\int_0^{x+at} u(x+at-y) F(dy) \right)\, dt\\&\qquad = \int_0^\infty \lambda e^{-\lambda t} \int_0^{x+at} v(x+at-y) F(dy) \, dt,
\end{align*}
where we used equality $\int_0^{x+at}F(dy) +\overline{F} (x+at)=1$.
Hence $v$ satisfies the equation
\begin{equation}\label{v10}
v(x)= \int_0^\infty \lambda e^{-\lambda t} \int_0^{x+at} v(x+at-y) F(dy) \, dt,
\end{equation}
which coincides with \cite[(1.19)]{EKM97}.
On the other hand, \eqref{v10} can be written in the form \cite[(1.22)]{EKM97}
\begin{equation}\label{eqforv}
v(x) = \frac{\theta}{1+\theta} + \frac{1}{1+\theta}  \int_0^x v(x-y) \,F_I (dy),
\end{equation}
where $F_I (x) =\frac{1}{\mu} \int_0^x \overline{F}(y)dy$ is the integrated tail of $F$, $\theta: = \frac{a}{\lambda \mu} -1$. Equivalently,
\begin{equation}\label{v20}
u(x) = \rho \overline{F}_I(x)+ \rho \int_0^x u(x-y)  F_I(dy),
\end{equation}
where $\rho= \frac{1}{1+\theta}$.
Note that we can apply to the above equation  Theorem~\ref{p2} with
$F_I$ instead of $F$.  Note that this model is defined for $x>0$, i.e. we restrict $h(x)= \rho \overline{F}_I(x)$ to  $[0,\infty)$.
Under the  stronger assumption that $F_I$ is subexponential,
the asymptotic behaviour of the solution to this equation is well known (cf. \cite[Thm. 2.1, p. 302]{AsmAlbr}):
\begin{equation}\label{as1druin}
u(x) = \frac{\rho}{1-\rho} \overline{F}_I(x) (1+o(1)), \quad x\to \infty.
\end{equation}
\end{example}

\section{Applications}\label{sec:app}

Properties of potentials of type \eqref{u10} are important in many applied probability models, such as
branching processes, queueing theory, insurance ruin theory, reliability theory, demography,
etc. \\
{\bf Renewal equation \eqref{ren11} and the one-dimensional random walk.}
Most of applications concern the renewal function
$u(x)=\Ee^0L_x$
where $L$ is a renewal process with the distribution $G$ of inter-arrival times.
In this case the renewal equation \eqref{ren11} holds true with
$ h(x)=G(x)$. For example, in {\it demographic models} (such as modelling Geiger counter or in a branching theory)
$L_x$ corresponds to the number of organisms/particles alive at time $x$; see for example \cite{Willmotetal, YZ06}.

Other applications come from
the distribution of all-time supremum $S=\max_{n\geq 1} S_n$ of a one-dimensional
random walk $S_n=\sum_{k=1}^n \eta_k$ (and $S_0=0$)  with
$\eta_k\geq 0$ and
\begin{equation}\label{eta}
\rho=\int_0^\infty\Pp(\eta_1 \in dz)<1.
\end{equation}
In this case the  function $v(x)=\Pp^0(S\leq x)$  for $x\geq 0$  satisfies the equation (cf. \cite[Prop. 2.9, p. 149]{As03})
\[v(x)=1-\rho+\rho\int_0^xv(x-y)G_\rho(dy)\]
with $G(dy)=\Pp(\eta_1\in dy)$ and  the proper distribution function $G_\rho(dy)=G(dy)/\rho$. Hence $u(x)=1-v(x)=\Pp^0(S>x)$ satisfies  the
equation
\[
u(x)=\rho\overline{G}_\rho(x)+\rho\int_0^xu(x-y)G_\rho(dy),
\]
which is \eqref{ren11}
with $h(x)= \rho\overline{G}_\rho(x)$.
As it is proved in \cite[Thm. 2.2, p. 224]{As03}, in case of a general non-defective random walk with negative drift,
one can take the first ascending ladder
height for the distribution of $\eta_1$.
In particular, in the case of a single server queue $GI|GI|1$ the quantity $S$ corresponds to the {\it steady-state workload}; see \cite[eq. (1.5), p. 268]{As03}.
Then $\eta_k$ are $k$th ascending ladder height of the
random walk $\sum_{k=1}^n\chi_k$ for $\chi_k$ being
the difference between successive i.i.d. service times $U_k$ and i.i.d. inter-arrival times $E_k$. In the case
of $M|G|1$ queue we have $\chi_k=U_k-E_k$, where $E_k$ is exponentially distributed with intensity, say, $\lambda$.
Then
\begin{equation}\label{ladder}
G(dx)=\mathbb{P}(\eta_1\in dx)= \lambda \mathbb{P}(U_1\leq x)dx;
\end{equation}
see \cite[Thm. 5.7, p. 237]{As03}.
Note that by \eqref{eta} in this case
$\rho=\lambda \Ee U_1$.
By  duality  (see e.g. \cite[Thm. 4.2, p. 261]{As03}), in the risk theory the tail distribution of $S$ corresponds to the ruin probability of a classical Cram\'er-Lundberg
process defined by
\begin{equation}\label{riskproc}
X_t=x+t-Z_t,
\end{equation}
where $Z_t=\sum_{i=1}^{N_t}U_k$
is
given in \eqref{CP} and describes a cumulative amount of the claims up to time $t$,
$N_t$ is a Poisson process with intensity $\lambda$ and $U_k$ is the claim size arrived at the $k$th epoch of the Poisson process $N$.
Here $x$ describes the initial capital of the insurance company and $a$ is a premium intensity.
Indeed, taking $\chi_k=U_k-E_k$ with exponentially distributed $E_k$ with intensity $\lambda$
one can prove that for the ruin time
$$
T=\inf\{t\geq 0: X_t<0\}
$$
we have
\begin{equation}\label{duality}
u(x)=\Pp^x(T<+\infty)=\Pp^0(S>x).
\end{equation}
 Note that by duality the service times $U_k$ in $GI|GI|1$ queue correspond to the claim sizes
and therefore we use the same letter to denote them.
Similarly, inter-arrival times $E_k$ in single server queue correspond to the times between
Poisson epochs of the process $N_t$ in risk process \eqref{riskproc}.
Assume  that $\delta=0$ in \eqref{bounddelta} and that $Y_s=s$, that is, $a=1$ in Example \ref{risk}.
If the net profit condition $\rho<1$ hold  true  (under which the above ruin probability is strictly less than one),
we can conclude that the ruin probability satisfies equation \eqref{v20}.
Hence from \cite[Thm. 5.2, p. 106]{FKZ13}, under the assumption that $F_I\in \Ss$ (which is equivalent to the assumption that $G\in \Ss$)
we derive the asymptotic   of the ruin probability given in \eqref{as1druin}.


{\bf Multivariate risk process.} There is an obvious need to understand the heavy-tailed asymptotic for the ruin probability
in the multi-dimensional set-up.
Consider  the multivariate risk process $X_t=(X_t^1, \ldots, X_t^d)$ with possibly dependent components $X_t^i$
describing the reserves of the $i$th insurance company which  covers incoming claims.
We assume that the  claims arrive simultaneously to all companies, that is,
$X_t$ is a multivariate L\'evy risk process with $a\in \mathbb{R}^d$,
$Z_t$ is a compound Poison process given in \eqref{CP} with the arrival intensity $\lambda$ and the generic claim size $U \in \mathbb{R}^d$.
We assume that $\delta=0$ and $Y_s=as$.
Each company can have  its own claims process as well. Indeed, to do so it suffices to
merge the separate independent arrival Poisson processes with the simultaneous arrival process (hence constructing new Poisson arrival process)
and allow the claim size to have atoms in one of  the axes directions.
Consider now the following
ruin time
$$
T= \inf\{t\geq 0: X_t\notin [0,\infty)^d \},
$$
which is the first exit time of $X$ from a non-negative quadrant, that is, $T$ is the first time when at least one company gets ruined.
Assume  the net profit condition
$\lambda \Ee U^{(k)}<1$ ($k=1,2,\ldots, d$) for $k$th coordinate  $U^{(k)}$ of the  generic claim size $U_1$.
Then from the compensation formula given in \cite[Thm. 3.4, p. 18]{AndreasGerber} (see also \cite[Eq. (5.5), p. 42]{AndreasGerber})
it follows that
$$
\Pp^x(\tau<\infty)=u(x)=\Ee^x \int_0^\infty l(X_s^\sharp)ds
$$
with $x=(x_1,\ldots, x_d)\in \rdp$ and
\begin{equation}\label{GS2}
l(x)=\lambda \int_{[x, \infty)}F(dz)=\lambda \overline{F}(x),
\end{equation}
where $F$ is the  claim size distribution.
In fact,  a  more general Gerber-Shiu function
\begin{equation}\label{GS}
u(x)=\Ee^x [e^{-q\tau}w(X_{T-}, |X_T|), \tau<\infty]
\end{equation}
can be represented as a potential function with
$$
l(z)=\lambda\int_z^\infty w(z,u-z)F(du);
$$
see \cite{Feng}.
The so-called penalty function $w$  in \eqref{GS} is applied to the deficit $X_T$ at the ruin moment and position $X_{T-}$ prior to the ruin time.

If $d=1$, then by \eqref{ruinell} and \eqref{GS2}
we recover heavy-tailed asymptotic of $u$ from Example \ref{risk}.

If $d=2$ (we have two companies) then using similar arguments to those in  Example \ref{risk}
for $v(x)=1-u(x)$ and $x=(x_1, x_2)\in \mathbb{R}^2_+$ we get
\begin{equation}\label{v10b}
v(x)= \int_0^\infty \lambda e^{-\lambda t} \int_{y_1\leq x_1+a_1 t, y_2\leq x_2+a_2 t} v(x+at-y) F(dy) \, dt,
\end{equation}
where $a=(a_1,a_2)$ and $y=(y_1,y_2)$.

Assume now that the claims coming simultaneously to both companies are independent on each other, that is $U_1=(U^{(1)}, U^{(2)})$
and $U^{(k)}$ are independent of each other with the distribution $F_k$ ($k=1,2$). Then equation \eqref{v10b} is equivalent to
\begin{equation*}
v(x)= \int_0^\infty \lambda e^{-\lambda t} \int_{0}^{x_1+a_1 t}\int_0^{x_2+a_2 t} v(x+at-y) F_2(dy_2)\, F_1(dy_1) \, dt.
\end{equation*}
Following Foss et al. \cite{FKP} we can also consider the proportional reinsurance
where the generic claim $U$ is divided into fixed proportion into two companies, that is
$U^{(2)}=\beta Z$ and $U^{(2)}=(1-\beta) Z$ for some random variable with distribution $F_Z$
and $\beta\in (0,1)$. In this case
\begin{equation*}
v(x)= \int_0^\infty \lambda e^{-\lambda t} \int_{0}^{(x_1+a_1 t) \wedge (x_2+a_2 t)} v\left(x+at-(\beta, 1-\beta) z\right) F_Z(dz) \, dt.
\end{equation*}
 Let $a_1>a_2$ and $x_1<x_2$. In this case by \cite[Cor. 2.1 and Cor. 2.2]{FKP}
we have
\[v(x)\sim \int_0^\infty \overline F_Z\left(\min\left\{x_1+\left(\frac{a_1}{\lambda}-\beta \Ee Z \right) t,  x_2+\left(\frac{a_2}{\lambda}-(1-\beta) \Ee Z \right)t \right\}\right)dt\]
as  $x^0\to\infty$ where $Z$ is strong subexponential, that is, $F_Z\in \Ss$ and
\begin{equation*}
  \int_0^b\overline{F}_Z(b-y)\overline{F}_Z(y)\,dy
  \sim 2 \Ee Z\overline{F}_Z(b)\quad\text{as }b \to \infty.
\end{equation*}
{\bf Mathematical finance.}
Other applications of the potential
function \eqref{u10} come from the mathematical finance.
For example, the renewal equation \eqref{ren1} can be used
in  pricing a perpetual put option; see Yin, Zhao \cite[Ex. 4.2]{YZ06}
for details.

The potential function appears in a consumption-investment problem initiated by Merton \cite{Merton} as well.
Consider a very simple model where on the market we have $d$ assets $S_t^i=e^{-X_t^i}$,  $1\leq i\leq d$,
governed by  exponential L\'evy processes $X_t^i$ (possibly depend on each other).
In fact, take
$
X_t=x+W_t-Z_t
$
with $W_t$ being $d$-dimensional Wiener process,
and $Z$ is defined in \eqref{CP}.
Let $(\pi_1,\pi_2,\ldots,\pi_d)$ be the strictly positive proportions of its total wealth that are invested in each of the $d$ stocks.
Then the wealth process equals
$
\sum_{i=1}^d \pi_iS_t^i.
$
Assume that the investor withdraws the proportion $\varpi$ of his funds for consumption.
The discounted utility of consumption
is measured by the function
$$
u(x)=\Ee^x\int_0^\infty e^{-qt} \ell (X_t)dt =\Ee^x\int_0^\infty \ell(X_s^\sharp)ds,
$$
where $q>0$, $T$ is an independent killing time exponentially distributed with parameter $q$
and
$$
\ell(x_1,x_2,\ldots,x_d)=L\left(\varpi\sum_{i=1}^d \pi_i e^{-x_i}\right)
$$
for some utility function $L$; see also \cite{Cadenillas} for details.
We take power utility $L(z)=z^\alpha$ for $\alpha \in (0,1)$
and $z>0$.
Assume that $F\in WS(\rd)$.
Since  $\ell (b x)\leq C\sum_{i=1}^d e^{-\alpha b_i x_i}$
for sufficiently large  constant $C$, we have
$\lim_{x^0\to \infty} \frac{\ell (x)}{\overline{F}( x)} =0$, and since $Y_t$ is a Wiener process,
$\lim_{x^0\to \infty} \frac{\overline{G}_\rho (x)}{\overline{F}(x)} =1$.
Hence by Theorem~\ref{p2} the asymptotic behaviour  of the discounted utility consumption is  $u(x)=o(1)\overline{F}(x)$
as $x^0\to \infty$ (that is, when initial assets prices go to zero).

 We choose only few examples where the subexponential asymptotic can be used
but the set of possible applications is much wider.

\end{document}